\documentclass{aptpub}

\authornames{David Cheek and Seva Shneer} % insert the authors here for use in running head

\shorttitle{Mean position of a branching L\'evy process}
\usepackage{amsfonts,amssymb,amsmath}
\usepackage{graphicx}
\usepackage{psfrag}
\usepackage{calc}
\usepackage{epic}
\usepackage{hyperref}
\usepackage{float}
\usepackage{latexsym}
\usepackage{subfigure,dsfont}
\allowdisplaybreaks

\usepackage{amssymb}

\begin{document}

\title{The empirical mean position \\ of a branching L\'evy process} % insert title - use \\ if it requires more than one line.

\authorone[Harvard University]{David Cheek}

\addressone{Program for Evolutionary Dynamics, Harvard University, Cambridge, Massachusetts, 
USA, dmcheek@g.harvard.edu}

\authortwo[Heriot-Watt University]{Seva Shneer} % Affiliation is just the name of your university or institution

\addresstwo{School of MACS, Heriot-Watt University, Edinburgh, EH14 4AS, United Kingdom; v.shneer@hw.ac.uk} % Your postal address goes here.

\begin{abstract}
We consider a supercritical branching L\'evy process on the real line. Under mild moment assumptions on the number of offspring and their displacements, we prove a second-order limit theorem on the empirical mean position.
\end{abstract}
\keywords{branching L\'evy process; branching random walk; branching Brownian motion; empirical mean position; empirical mean depth}
\ams{60J80}{60J25}

\section{Introduction}
A branching L\'evy process describes a population of particles undergoing spatial movement, death, and reproduction.  It can be defined informally as follows (for a formal definition, see Section \ref{sec:model}). Initially there is one particle located at the origin of the real line. The particle lives for an exponentially-distributed time. During this time it moves according to a L\'evy process. At the time of death, the particle is replaced by a random number of new particles, displaced from the parent particle's death position according to a point process. All particles move, die, and reproduce in a statistically identical manner, independently of every other particle. We are only concerned with the supercritical case. That is, each particle gives birth to more than one particle on average, and thus the total number of particles grows to infinity with positive probability.

The particle positions' empirical distribution has received much attention, especially for branching random walks and branching Brownian motion, which are special cases of the model. There are many results on the empirical distribution's maximum \cite{bramson1978maximal,biggins1995growth,gantert2018large}, as well as on large deviations \cite{louidor2017large} and on the almost-sure weak convergence to a Gaussian distribution \cite{biggins1990central,gao2018second}.

The empirical mean position, which is simple and important for applications, has received relatively little attention. For specific branching random walks, \cite{meli2019sample} shows that the empirical mean position almost surely grows asymptotically linearly with time, while \cite{duffy2019variance} shows that the empirical mean position's variance converges. These results combined raise the question of characterising a second-order limit term.

For branching L\'evy processes, under some mild moment assumptions on the number of offspring and their displacements, we prove a second-order limit theorem for the empirical mean position. Namely, we show that the difference between the empirical mean position at time $t$ and $rt$, for some constant $r$, converges almost surely to a random variable.

Before proceeding with the remainder of the paper, we discuss some special cases of the model and applications.

First, consider that particles do not move during their lifetime and that each particle is displaced by $+1$ from its parent. A particle's position is its generation. Our result describes the average generation, complementing results of \cite{meli2019sample,chauvin2005martingales}. Second, consider instead that displacement sizes are Poisson distributed. This is a popular model for cancer evolution \cite{durrett2013population}. Here particles are cells, and a cells' position is its number of mutations. Our result gives the average number of mutations per cell. Third, consider that particles are not displaced from their parent but move as a random walk during their lifetime. This model is seen in phylogenetics. The branching process represents speciation  \cite{aldous2001stochastic}, while the positions are lengths of a particular DNA segment \cite{felsenstein2004phylogenetics}.

The remainder of the paper is organised as follows. We  introduce the model in Section \ref{sec:model}, formulate our main result in Section \ref{sec:main} and prove it in Section \ref{sec:proof}.

\section{Model} \label{sec:model}
Initially there is a single particle named $\emptyset$ which moves according to a L\'evy process $(Z_{\emptyset,s})_{s\geq0}$, with $Z_{\emptyset,0}=0$ and $\mathbb{E}[Z_{\emptyset,1}^2]<\infty$. After an exponentially distributed waiting time $A_\emptyset$, the particle dies and is replaced by a random number $N_\emptyset$ of new particles with $\mathbb{E}[N_\emptyset]>1$ and $\mathbb{E}[N_\emptyset^2]<\infty$. The new particles are born at positions $(Z_{\emptyset,A_\emptyset}+D_i)_{i=1}^{N_\emptyset}$. The $D_i$ are $\mathbb{R}$-valued random variables with
$$\mathbb{E}\left[\left(\sum_{i=1}^{N_\emptyset}D_i\right)^2\right]<\infty\quad\text{and}\quad\mathbb{E}\left[\sum_{i=1}^{N_\emptyset}D_i^2\right]<\infty.$$
Independence is assumed between $(Z_{\emptyset,s})_{s\geq0}$ and $A_\emptyset$ and $(D_i)_{i=1}^{N_\emptyset}$ (but the $D_i$ need not be independent of each other nor of $N_\emptyset$). All particles independently follow the initial particle's behaviour.

To denote particles we follow standard notation. Let
$$
\mathcal{T}=\bigcup_{n\in\mathbb{N}\cup\{0\}}\mathbb{N}^n.
$$
Here $\mathbb{N}^0=\{\emptyset\}$ contains the initial particle. For $v=(v_1,..,v_n)\in\mathcal{T}$ and $i\in\mathbb{N}$ write $vi=(v_1,..,v_n,i)$, where $v$ is the parent of $vi$. To describe genealogical relationships, the set $\mathcal{T}$ is endowed with a partial ordering $\prec$, defined by
$$
(u_i)_{i=1}^m\prec(v_i)_{i=1}^n\iff m<n \text{ and }(u_i)_{i=1}^m=(v_i)_{i=1}^m.
$$
Write $\preceq$ for $\prec$ or $=$.

Now let
$$
\left[(Z_{v,s})_{s\geq0},A_v,(D_{vi})_{i=1}^{N_v}\right]
$$
for $v\in\mathcal{T}$ be i.i.d. copies of
$$
\left[(Z_{\emptyset,s})_{s\geq0},A_\emptyset,(D_i)_{i=1}^{N_\emptyset}\right].
$$
The set of all particles to ever exist is
$$
\mathcal{T}^*=\left\{(v_i)_{i=1}^n\in\mathcal{T}:v_{m+1}\leq N_{(v_i)_{i=1}^m},\text{ for }m=0,1,..,n-1\right\}.$$
The particles alive at time $t\geq0$ are
$$
\mathcal{T}_t=\left\{v\in\mathcal{T}^*:\sum_{u\prec v}A_u\leq t<\sum_{u\preceq v}A_u\right\}.
$$
Particle $v$ at time $t$, if it is alive, has position
$$
X_{v,t}=\sum_{\emptyset\prec u\preceq v}D_u+\sum_{\emptyset\preceq u\prec v}Z_{u,A_u}+Z_{v,t-\sum_{\emptyset\preceq u\prec v}A_u}.
$$
For further notation, the branching rate is $$\lambda=\mathbb{E}[A_\emptyset]^{-1},$$ the effective branching rate is $$\hat{\lambda}=\lambda\mathbb{E}[N_\emptyset-1],$$ and the movement rate is
$$
r=\mathbb{E}[Z_{\emptyset,1}]+\lambda\mathbb{E}\left[\sum_{i=1}^{N_\emptyset}D_i\right].
$$
\section{Main result}\label{sec:main}
\begin{theorem}\label{mr}Conditional on the event $\{\lim_{t\rightarrow\infty}|\mathcal{T}_t|=\infty\}$, the limit
$$
\lim_{t\rightarrow\infty}\frac{1}{|\mathcal{T}_t|}\left(\sum_{v\in\mathcal{T}_t}X_{v,t}-rt\right)
$$
exists and is finite almost surely.
\end{theorem}

\section{Proof of Theorem \ref{mr}} \label{sec:proof}
Our proof will involve conditioning on whether branching occurs during the time interval $[0,h]$ for some small $h>0$. Write
$$
J_{0,h}=\{A_\emptyset>h\}
$$
for the event that the first branching occurs after time $h$. Write
$$
J_{1,h}=\left\{A_\emptyset\leq h<A_\emptyset+\min_{i=1,..,N_\emptyset} A_i\right\}
$$
for the event that the first branching occurs before time $h$ and the second branching occurs after time $h$.
Write
$$
J_{2,h}=\left\{A_\emptyset+\min_{i=1,..,N_\emptyset} A_i\leq h\right\}.
$$
for the event that the second branching occurs before time $h$.
Note the probabilities
$$
\begin{cases}
\mathbb{P}[J_{0,h}]=1-h\lambda+o(h)\\
\mathbb{P}[J_{1,h}]=h\lambda+o(h)\\
\mathbb{P}[J_{2,h}]=o(h),
\end{cases}
$$
as $h\downarrow0$.
Observe the conditional distribution
\begin{equation}\label{cd1}
\left(\sum_{v\in\mathcal{T}_{t+h}}\left(X_{v,t+h}-r(t+h)\right)\big|J_{0,h}\right)\overset{d}{=}\sum_{v\in\mathcal{T}_t'}(Z_{\emptyset,h}+X_{v,t}'-r(t+h)),
\end{equation}
where $(X_{v,t}')_{v\in\mathcal{T}_t'}\overset{d}{=}(X_{v,t})_{v\in\mathcal{T}_t}$, and $(X_{v,t}')_{v\in\mathcal{T}_t'}$ is independent of $Z_{\emptyset,h}$. Meanwhile
\begin{equation}\label{cd2}
\left(\sum_{v\in\mathcal{T}_{t+h}}\left(X_{v,t+h}-r(t+h)\right)\big|J_{1,h}\right)\overset{d}{=}\sum_{i=1}^{N_\emptyset}\sum_{v\in\mathcal{T}^i_t}(D_i+X^i_{v,t}-rt)+\eta_h,
\end{equation}
where $(X^i_{v,t})_{v\in\mathcal{T}_t^i}\overset{d}{=}(X_{v,t})_{v\in\mathcal{T}_t}$ for $i=1,..,N_\emptyset$; the $(X^i_{v,t})_{v\in\mathcal{T}_t^i}$ are independent of each other and of $\left(D_i\right)_{i=1}^{N_\emptyset}$; and
$$\eta_{h}=\left(\sum_{i=1}^{N_{\emptyset}}|\mathcal{T}_t^i|(Z_{\emptyset,A_{\emptyset}} + Z_{i,h-A_{\emptyset}}-rh)|J_{1,h}\right).$$
Straightforward calculations show that the first and second moments of $\eta_h$ converge to $0$ as $h\downarrow0$.
\begin{lem}\label{zeroexp}For $t\geq0$,
$$
\mathbb{E}\left[\sum_{v\in\mathcal{T}_t}\left(X_{v,t}-rt\right)\right]=0.
$$
\begin{proof}
From (\ref{cd1}),
$$
\mathbb{E}\left[\sum_{v\in\mathcal{T}_{t+h}}\left(X_{v,t+h}-r(t+h)\right)|J_{0,h}\right]=\mathbb{E}\left[\sum_{v\in\mathcal{T}_t}\left(X_{v,t}-rt\right)\right]+h\left(\mathbb{E}[Z_{\emptyset,1}]-r\right)\mathbb{E}|\mathcal{T}_t|.
$$
From (\ref{cd2}),
\begin{align*}
\mathbb{E}\left[\sum_{v\in\mathcal{T}_{t+h}}\left(X_{v,t+h}-r(t+h)\right)|J_{1,h}\right]=&\mathbb{E}[N_\emptyset ]\mathbb{E}\left[\sum_{v\in\mathcal{T}_t}\left(X_{v,t}-rt\right)\right]\\&+\mathbb{E}\left[\sum_{i=1}^{N_\emptyset} D_i\right]\mathbb{E}|\mathcal{T}_t|+o(1).
\end{align*}
Taking the unconditional expectation,
\begin{align*}
\mathbb{E}\left[\sum_{v\in\mathcal{T}_{t+h}}\left(X_{v,t+h}-r(t+h)\right)\right]  =& (1-h\lambda)\mathbb{E}\left[\sum_{v\in\mathcal{T}_{t+h}}\left(X_{v,t+h}-r(t+h)\right)|J_{0,h}\right]\\
&+h\lambda\mathbb{E}\left[\sum_{v\in\mathcal{T}_{t+h}}\left(X_{v,t+h}-r(t+h)\right)|J_{1,h}\right]
+o(h)\\
=&\mathbb{E}\left[\sum_{v\in\mathcal{T}_t}\left(X_{v,t}-rt\right)\right](1+h\hat{\lambda})+o(h).
\end{align*}
Rearranging and taking $h\downarrow0$,
\begin{eqnarray*}
\frac{d}{dt}\mathbb{E}\left[\sum_{v\in\mathcal{T}_t}\left(X_{v,t}-rt\right)\right]=\hat{\lambda}\mathbb{E}\left[\sum_{v\in\mathcal{T}_t}\left(X_{v,t}-rt\right)\right].
\end{eqnarray*}
The statement of the lemma for any $t$ now follows from the above and the fact that it clearly holds for $t=0$.
\end{proof}
\end{lem}
Next we determine second moments.
\begin{lem}\label{bsm}For $t\geq0$,
$$
\mathbb{E}\left[\left(\sum_{v\in\mathcal{T}_t}(X_{v,t}-rt)\right)^2\right]=c_1e^{2\hat{\lambda}t}-c_2te^{\hat{\lambda}t}-c_1e^{\hat{\lambda}t},$$
where
$$c_1=\frac{\mathbb{E}[(N_\emptyset-1)^2]}{\mathbb{E}[N_\emptyset-1]^2}\mathbb{E}\left[\sum_{i=1}^{N_\emptyset}D_i^2\right]+\frac{1}{\mathbb{E}[N_\emptyset-1]}\mathbb{E}\left[\left(\sum_{i=1}^{N_\emptyset}D_i\right)^2\right]$$
and
$$
c_2=\hat{\lambda}\frac{\mathbb{E}[(N_\emptyset-1)^2]}{\mathbb{E}[N_\emptyset-1]^2}\mathbb{E}\left[\sum_{i=1}^{N_\emptyset}D_i^2\right].
$$
\end{lem}
\begin{proof}
From (\ref{cd1}),
\begin{align*}
\mathbb{E}\left[\left(\sum_{v\in\mathcal{T}_{t+h}}(X_{v,t+h}-r(t+h))\right)^2\Big| J_{0,h}\right]=&\mathbb{E}\left[\left(\sum_{v\in\mathcal{T}_t}(X_{v,t}-rt)\right)^2\right]\\&+2h\left(\mathbb{E}[Z_{\emptyset,1}]-r\right)\mathbb{E}\left[|\mathcal{T}_t|\sum_{v\in\mathcal{T}_t}(X_{v,t}-rt)\right]\\&+h\left(\mathbb{E}[Z_{\emptyset,1}^2]-\mathbb{E}[Z_{\emptyset,1}]^2\right)\mathbb{E}[|\mathcal{T}_t|^2]\\&+o(h).
\end{align*}
From (\ref{cd2}) and Lemma \ref{zeroexp},
\begin{align*} \mathbb{E}\left[\left(\sum_{v\in\mathcal{T}_{t+h}}(X_{v,t+h}-r(t+h))\right)^2\Big| J_{1,h}\right] =&\mathbb{E}\left[\left(\sum_{i=1}^{N_\emptyset}\sum_{v\in\mathcal{T}^i_t}(X_{v,t}-rt)\right)^2\right]
\\&+2\mathbb{E}\left[\sum_{i=1}^{N_\emptyset}D_i|\mathcal{T}_t^i|\sum_{v\in\mathcal{T}^i_t}(X_{v,t}-rt)\right]\\ & +2\mathbb{E}\left[\sum_{\substack{i,j=1\\i\not=j}}^{N_\emptyset}D_i|\mathcal{T}_t^i|\sum_{v\in\mathcal{T}^j_t}(X_{v,t}-rt)\right]\\&+\mathbb{E}\left[\left(\sum_{i=1}^{N_\emptyset}D_i|\mathcal{T}_t^i|\right)^2\right]+o(1)\\
=&\mathbb{E}[N_\emptyset]\mathbb{E}\left[\left(\sum_{v\in\mathcal{T}_t}(X_{v,t}-rt)\right)^2\right]\\&+2\mathbb{E}\left[\sum_{i=1}^{N_\emptyset}D_i\right]\mathbb{E}\left[|\mathcal{T}_t|\sum_{v\in\mathcal{T}_t}(X_{v,t}-rt)\right]\\
&+\mathbb{E}\left[\sum_{i=1}^{N_\emptyset}D_i^2\right]\left(\mathbb{E}\left[|\mathcal{T}_t|^2\right]-\left(\mathbb{E}|\mathcal{T}_t|\right)^2\right)\\&+\mathbb{E}\left[\left(\sum_{i=1}^{N_\emptyset}D_i\right)^2\right]\left(\mathbb{E}|\mathcal{T}_t|\right)^2+o(1).
\end{align*}

But $\mathbb{E}|\mathcal{T}_t|$ and $\mathbb{E}[|\mathcal{T}_t|^2]$ are standard knowledge \cite{athreya1972branching}:
$$
\mathbb{E}|\mathcal{T}_t|=e^{\hat{\lambda}t}
$$
and
$$
\mathbb{E}[|\mathcal{T}_t|^2]=\left(1+\frac{\mathbb{E}[(N_\emptyset-1)^2]}{\mathbb{E}[N_\emptyset-1]}\right)e^{2\hat{\lambda}t}-\frac{\mathbb{E}[(N_\emptyset-1)^2]}{\mathbb{E}[N_\emptyset-1]}e^{\hat{\lambda}t}.
$$
Therefore
\begin{align*}
 &\mathbb{E}\left[\left(\sum_{v\in\mathcal{T}_{t+h}}(X_{v,t+h}-r(t+h))\right)^2\right] \\&=(1-\lambda h)\mathbb{E}\left[\left(\sum_{v\in\mathcal{T}_{t+h}}(X_{v,t+h}-r(t+h))\right)^2\Big| J_{0,h}\right]\\ &\quad +h\lambda\mathbb{E}\left[\left(\sum_{v\in\mathcal{T}_{t+h}}(X_{v,t+h}-r(t+h))\right)^2\Big| J_{1,h}\right]+o(h)\\
&=1+h\hat{\lambda}\mathbb{E}\left[\left(\sum_{v\in\mathcal{T}^i_t}(X_{v,t}-rt)\right)^2\right]\\&\quad+h ae^{2\hat{\lambda} t}+hb e^{\hat{\lambda} t}+o(h),
\end{align*}
where
$$
a=\lambda\left(\frac{\mathbb{E}[(N_\emptyset-1)^2]}{\mathbb{E}[N_\emptyset-1]}\mathbb{E}\left[\sum_{i=1}^{N_\emptyset}D_i^2\right]+\mathbb{E}\left[\left(\sum_{i=1}^{N_\emptyset}D_i\right)^2\right]\right)
$$
and
$$
b=-\lambda\frac{\mathbb{E}[(N_\emptyset-1)^2]}{\mathbb{E}[N_\emptyset-1]}\mathbb{E}\left[\sum_{i=1}^{N_\emptyset}D_i^2\right]
$$
Rearranging and taking $h\downarrow0$,
\begin{align*}
 \frac{d}{dt}\mathbb{E}\left[\left(\sum_{v\in\mathcal{T}^i_t}(X_{v,t}-rt)\right)^2\right]  = \hat{\lambda}\mathbb{E}\left[\left(\sum_{v\in\mathcal{T}^i_t}(X_{v,t}-rt)\right)^2\right]+ae^{2\hat{\lambda}t}+be^{\hat{\lambda}t}.
\end{align*}
The statement of Lemma \ref{bsm} now follows directly from the differential equation above.
\end{proof}
Next we present a martingale result for which a filtration $(\mathcal{F}_t)_{t\geq0}$ needs to be defined:
$$\mathcal{F}_t=\sigma\left((X_{v,s})_{v\in\mathcal{T}_s}:0\leq s\leq t\right).$$
\begin{lem}\label{mart}$$\left(e^{-\hat{\lambda} t}\sum_{v\in\mathcal{T}_t}\left(X_{v,t}-rt\right)\right)_{t\geq0}$$
is a martingale with respect to $(\mathcal{F}_t)_{t\geq0}$.
\end{lem}
\begin{proof}
Write
$$
\mathcal{T}_{u,t}=\{v\in\mathcal{T}_t:u\preceq v\}
$$
for the particles alive at time $t$ which are descendants of $u\in\mathcal{T}$. Let $0\leq s\leq t$. Then
\begin{align*}
e^{-\hat{\lambda} t}\sum_{v\in\mathcal{T}_t}\left(X_{v,t}-rt\right)=&e^{-\hat{\lambda} t}\sum_{u\in\mathcal{T}_s}\sum_{v\in\mathcal{T}_{u,t}}\left(X_{v,t}-X_{u,s}-r(t-s)\right)\\&+e^{-\hat{\lambda} t}\sum_{u\in\mathcal{T}_s}|\mathcal{T}_{u,t}|\left(X_{u,s}-rs\right).
\end{align*}
Taking conditional expectations,
\begin{align*}
& \mathbb{E}\left[e^{-\hat{\lambda} t}\sum_{v\in\mathcal{T}_t}\left(X_{v,t}-rt\right)\big|\mathcal{F}_s\right]\\ &= e^{-\hat{\lambda} t}|\mathcal{T}_s|\mathbb{E}\left[\sum_{v\in\mathcal{T}_{t-s}}(X_{v,t-s}-r(t-s))\right]+e^{-\hat{\lambda} t}\sum_{u\in\mathcal{T}_s}e^{\hat{\lambda}(t-s)}(X_{u,s}-rs)\\
&=e^{-\hat{\lambda} s}\sum_{u\in\mathcal{T}_s}(X_{u,s}-rs),
\end{align*}
where the last equality is due to Lemma \ref{zeroexp}.
\end{proof}
\begin{proof}[Proof of Theorem \ref{mr}]
By Lemmas \ref{bsm} and \ref{mart} and the martingale convergence theorem, there is a $\mathbb{R}$-valued random variable $V$ with
\begin{equation}\label{mcr}
\lim_{t\rightarrow\infty}e^{-\hat{\lambda} t}\sum_{v\in\mathcal{T}_t}\left(X_{v,t}-rt\right)=V
\end{equation}
almost surely. But conditioned on the event $\{\lim_{t\rightarrow\infty}|\mathcal{T}_t|=\infty\}$, there is a positive random variable $W$ with
\begin{equation}\label{anco}
\lim_{t\rightarrow\infty}e^{-\hat{\lambda} t}|\mathcal{T}_t|=W
\end{equation}
almost surely  \cite{athreya1972branching}. Combine (\ref{mcr}) and (\ref{anco}) to conclude the proof.
\end{proof}

\section*{Acknowledgments}
The authors are grateful to Ken Duffy for a number of useful discussions and to the anonymous referee for a number of useful comments on the first version of our paper.

%\bibliographystyle{abbrv}
%\bibliography{branching}

\begin{thebibliography}{10}

\bibitem{aldous2001stochastic}
D.~J. Aldous et~al.
\newblock Stochastic models and descriptive statistics for phylogenetic trees,
  from {Y}ule to today.
\newblock {\em Statistical Science}, 16(1):23--34, 2001.

\bibitem{athreya1972branching}
K.~Athreya and P.~Ney.
\newblock Branching processes, 1972.

\bibitem{biggins1990central}
J.~Biggins.
\newblock The central limit theorem for the supercritical branching random
  walk, and related results.
\newblock {\em Stochastic processes and their applications}, 34(2):255--274,
  1990.

\bibitem{biggins1995growth}
J.~Biggins.
\newblock The growth and spread of the general branching random walk.
\newblock {\em The Annals of Applied Probability}, pages 1008--1024, 1995.

\bibitem{bramson1978maximal}
M.~D. Bramson.
\newblock Maximal displacement of branching brownian motion.
\newblock {\em Communications on Pure and Applied Mathematics}, 31(5):531--581,
  1978.

\bibitem{chauvin2005martingales}
B.~Chauvin, T.~Klein, J.-F. Marckert, A.~Rouault, et~al.
\newblock Martingales and profile of binary search trees.
\newblock {\em Electronic Journal of Probability}, 10:420--435, 2005.

\bibitem{duffy2019variance}
K.~R. Duffy, G.~Meli, and S.~Shneer.
\newblock The variance of the average depth of a pure birth process converges
  to 7.
\newblock {\em Statistics \& Probability Letters}, 150:88--93, 2019.

\bibitem{durrett2013population}
R.~Durrett.
\newblock Population genetics of neutral mutations in exponentially growing
  cancer cell populations.
\newblock {\em The annals of applied probability: an official journal of the
  Institute of Mathematical Statistics}, 23(1):230, 2013.

\bibitem{felsenstein2004phylogenetics}
J.~Felsenstein.
\newblock Inferring phylogenies, 2004.

\bibitem{gantert2018large}
N.~Gantert, T.~H{\"o}felsauer, et~al.
\newblock Large deviations for the maximum of a branching random walk.
\newblock {\em Electronic Communications in Probability}, 23, 2018.

\bibitem{gao2018second}
Z.~Gao, Q.~Liu, et~al.
\newblock Second and third orders asymptotic expansions for the distribution of
  particles in a branching random walk with a random environment in time.
\newblock {\em Bernoulli}, 24(1):772--800, 2018.

\bibitem{louidor2017large}
O.~Louidor and E.~Tsairi.
\newblock Large deviations for the empirical distribution in the general
  branching random walk.
\newblock {\em arXiv preprint arXiv:1704.02374}, 2017.

\bibitem{meli2019sample}
G.~Meli, T.~S. Weber, and K.~R. Duffy.
\newblock Sample path properties of the average generation of a
  {B}ellman--{H}arris process.
\newblock {\em Journal of mathematical biology}, 79(2):673--704, 2019.

\end{thebibliography}

\end{document}